\documentclass[a4paper, 11pt]{amsart}

\usepackage{epsfig}
\usepackage{amsthm,amsfonts}
\usepackage{amssymb,graphicx,color}
\usepackage[all]{xy}
\usepackage{cancel} 
\usepackage{verbatim}
\usepackage{hyperref}
\usepackage{enumitem}
\usepackage{neuralnetwork}
\usepackage{tikz}
\usepackage{tikz-cd}
\usepackage{graphicx}
\usepackage{float}

\usetikzlibrary{arrows}

\newtheorem{theorem}{Theorem}[section]
\newtheorem*{theorem*}{Theorem}

\newtheorem{corollary}[theorem]{Corollary}
\newtheorem{proposition}[theorem]{Proposition}
\newtheorem{definition}[theorem]{Definition}
\newtheorem{example}[theorem]{Example}
\newtheorem{remark}[theorem]{Remark}

\newcommand{\R}{\mathbb{R}}

\newcommand{\N}{\mathbb{N}}

\graphicspath{{Images/}}

\begin{document}

\title[Classification of real algebraic curves]
{Classification of real algebraic curves under blow-spherical homeomorphisms at infinity}

\author[J. E. Sampaio]{Jos\'e Edson Sampaio}
\author[E. C. da Silva]{Euripedes Carvalho da Silva}

\address{Jos\'e Edson Sampaio:  
              Departamento de Matem\'atica, Universidade Federal do Cear\'a,
	      Rua Campus do Pici, s/n, Bloco 914, Pici, 60440-900, 
	      Fortaleza-CE, Brazil. \newline  
              E-mail: {\tt edsonsampaio@mat.ufc.br}                    
}

\address{Euripedes Carvalho da Silva: Departamento de Matem\'atica, Instituto Federal de Educa\c{c}\~ao, Ci\^encia e Tecnologia do Cear\'a,
 	      Av. Parque Central, 1315, Distrito Industrial I, 61939-140, 
 	      Maracana\'u-CE, Brazil. \newline  
               E-mail: {\tt euripedes.carvalho@ifce.edu.br}
 } 

\thanks{The first named author was partially supported by CNPq-Brazil grant 310438/2021-7. This work was supported by the Serrapilheira Institute (grant number Serra -- R-2110-39576).
}

\begin{abstract}

In this article, we present a complete classification, with normal forms, of the real algebraic curves under blow-spherical homeomorphisms at infinity.  
\end{abstract}

\maketitle

\tableofcontents

\section{Introduction}
In this article, we study real algebraic curves under blow-spherical homeomorphisms from the global point of view. 
Roughly speaking, two subsets of Euclidean spaces are blow-spherical homeomorphic, if their spherical modifications (see Definition \ref{def:strict_transf}) are homeomorphic and, in particular, this homeomorphism induces a homeomorphism between their tangent links (see Definition \ref{def:bs_homeomorphism}).
This gives an equivalence which lives between topological equivalence and semialgebraic bi-Lipschitz equivalence. 

The study of analytic sets under blow-spherical homeomorphisms from the local point of view has been studied in some works, e.g., \cite{BirbrairFG:2012,BirbrairFG:2017, Sampaio:2015,Sampaio:2020,Sampaio:2022b}.
In \cite{Sampaio:2020}, the first author, among other things, presented a complete classification of the germs of complex analytic curves under blow-spherical homeomorphisms. In \cite{Sampaio:2022b}, the first author presented several results of related to blow-spherical homeomorphism between germs of real analytic sets and, in particular, he presented a classification of germs of real analytic curves under blow-spherical homeomorphisms.

Recently, the authors of this article in \cite{SampaioS:2023} presented a complete classification of complex algebraic curves under (global) blow-spherical homeomorphisms (see Theorem 4.6 in \cite{SampaioS:2023}). They also presented a complete classification of complex algebraic curves under blow-spherical homeomorphisms at infinity with normal forms (see Theorems 4.2 and 4.3 in \cite{SampaioS:2023}). 

So, it becomes natural to try classifying real algebraic curves under blow-spherical homeomorphisms.

The main aim of this article is to present a complete classification of real algebraic curves under blow-spherical homeomorphisms at infinity (see Proposition \ref{propblowspherical}). Moreover, we also present normal forms for this classification (see Proposition \ref{thmnormalform}). We also present a result of realization of our (complete) invariant (see Proposition \ref{thm:realizable}). 

 In order to be a bit more precise, for a real algebraic curve $X$, we define our invariant $k(X,\infty)$ (see Definition \ref{def:relative_mult}), which is a point of $(\mathbb{Z}_{>0})^n$ and $n$ is the cardinality of the link (at 0) of the tangent cone of $X$ at infinity (see Subsection \ref{subsec:tg_cones}), and in Proposition \ref{propblowspherical}, we prove that two real algebraic curve $X$ and $\widetilde{X}$ are blow-spherical homeomorphic at infinity if and only if $k(X,\infty)=k(\widetilde{X},\infty)$.
 
 In Subsection \ref{sec:normal_forms}, we present a collection of real algebraic curves and we prove in Proposition \ref{thmnormalform} that a real algebraic curve is blow-spherical homeomorphic at infinity to exactly one curve of that collection.
 
In Subsection \ref{sec:realization}, we observe that $k(X,\infty)\in \mathcal{N}$, where $\mathcal{N}=\bigcup\limits_{n=1}^{\infty} \mathcal{N}_n$ and $\mathcal{N}_n$ is the set of all $(\eta_1, \eta_2,\cdots,\eta_n)\in (\mathbb{Z}_{>0})^n$ such that $\eta _1\leq \eta _2\leq \cdots \leq \eta_n$. Moreover, given $\eta=(\eta_1,...,\eta_k)\in \mathcal{N}$, we prove in Proposition \ref{thm:realizable} that there is a real algebraic curve $X$ such that $ k(X,\infty)=\eta$ if and only if $\eta_1+...+\eta_k$ is an even number.

\bigskip

\section{Preliminaries}\label{section:cones}
Here, all the real algebraic sets are supposed to be pure dimensional.

\subsection{Definition of the blow-spherical equivalence}
Let us consider the {\bf spherical blowing-up at infinity} (resp. $p$) of $\R^{n+1}$, $\rho_{\infty}\colon\mathbb{S}^n\times (0,+\infty )\to \R^{n+1}$ (resp. $\rho_p\colon\mathbb{S}^n\times [0,+\infty )\to \R^{n+1}$), given by $\rho_{\infty}(x,r)=\frac{1}{r}x$ (resp. $\rho_p(x,r)=rx+p$).

Note that $\rho_{\infty}\colon\mathbb{S}^n\times (0,+\infty )\to \R^{n+1}\setminus \{0\}$ (resp. $\rho_p\colon\mathbb{S}^n\times (0,+\infty )\to \R^{n+1}\setminus \{0\}$) is a homeomorphism with inverse mapping $\rho_{\infty}^{-1}\colon\R^{n+1}\setminus \{0\}\to \mathbb{S}^n\times (0,+\infty )$ (resp. $\rho_p \colon\mathbb{S}^n\times (0,+\infty )\to \R^{n+1}\setminus \{0\}$) given by $\rho_{\infty}^{-1}(x)=(\frac{x}{\|x\|},\frac{1}{\|x\|})$ (resp. $\rho_p^{-1}(x)=(\frac{x-p}{\|x-p\|},\|x-p\|)$).

\begin{definition}\label{def:strict_transf}
 The {\bf strict transform} of the subset $X$ under the spherical blowing-up $\rho_{\infty}$ is $X'_{\infty}:=\overline{\rho_{\infty}^{-1}(X\setminus \{0\})}$ (resp. $X'_{p}:=\overline{\rho_{p}^{-1}(X\setminus \{0\})}$). The subset $X_{\infty}'\cap (\mathbb{S}^n\times \{0\})$ (resp. $X_{p}'\cap (\mathbb{S}^n\times \{0\})$) is called the {\bf boundary} of $X'_{\infty}$ (resp. $X'_p$) and it is denoted by $\partial X'_{\infty}$ (resp. $\partial X'_p$). 
\end{definition}

\begin{definition}\label{def:bs_homeomorphism}
Let $X$ and $Y$ be subsets in $\mathbb{R}^n$ and $\mathbb{R}^m$, respectively. Let $p\in \mathbb{R}^n\cup \{\infty\}$, $q\in \mathbb{R}^m\cup \{\infty\}$. A homeomorphism $\varphi:X\rightarrow Y$ such that $q=\lim\limits_{x\rightarrow p}{\varphi(x)}$ is said a {\bf blow-spherical homeomorphism at $p$}, if  the homeomorphism 
$$\rho^{-1}_{q}\circ \varphi\circ \rho_p \colon X'_p\setminus \partial X'_p\rightarrow Y'_q\setminus \partial Y'_q$$
extends to a homeomorphism $\varphi'\colon X'_p\rightarrow Y'_q$. A homeomorphism $\varphi\colon X\rightarrow Y$ is said a {\bf blow-spherical homeomorphism} if it is a blow-spherical homeomorphism for all $p\in \overline{X}\cup \{\infty\}$. In this case, we say that the sets $X$ and $Y$ are {\bf blow-spherical homeomorphic} or {\bf blow-isomorphic} ({\bf at $(p,q)$}). 
\end{definition}

\begin{definition}\label{def:strong_diffeo}
Let $X$ and $Y$ be subsets in $\mathbb{R}^n$ and $\mathbb{R}^m$, respectively. We say that a blow-spherical homeomorphism $h\colon X\rightarrow Y$ is a {\bf strong blow-spherical homeomorphism} if $h({\rm Sing}_1(X))={\rm Sing}_1(Y)$ and $h|_{X\setminus {\rm Sing}_1(X)}\colon X\setminus {\rm Sing}_1(X)\rightarrow Y\setminus {\rm Sing}_1(Y)$ is a $C^1$ diffeomorphism, where, for $A\subset \R^p$, ${\rm Sing}_k(A)$ denotes the points $x\in A$ such that, for any open neighbourhood $U$ of $x$, $A\cap U$ is not a $C^k$ submanifold of $\R^p$. A blow-spherical homeomorphism at $\infty$, $\varphi\colon X\rightarrow Y$, is said a {\bf strong blow-spherical homeomorphism at $\infty$ } if there are compact sets $K\subset \mathbb{R}^n$ and $\tilde{K}\subset\mathbb{R}^m$ such that $\varphi(X\setminus K)=Y\setminus \tilde{K}$ and the restriction $\varphi|_{X\setminus K}\colon X\setminus K\to Y\setminus \tilde{K}$ is a strong blow-spherical homeomorphism. 
\end{definition}

\begin{remark}
 We have some examples:
 \begin{enumerate}
  \item $\mbox{Id}:X \rightarrow X$ is a blow-spherical homeomorphism for any $X\subset \mathbb{R}^n$;
  \item Let $X\subset \mathbb{R}^n$, $Y\subset \mathbb{R}^m$ and $Z\subset \mathbb{R}^k$ be subsets. If $f\colon X\rightarrow Y$ and $g\colon Y\rightarrow Z$ are blow-spherical homeomorphisms, then $g\circ f\colon X\rightarrow Z$ is a blow-spherical homeomorphism.
 \end{enumerate}

\end{remark}

Thus we have a category called {\bf blow-spherical category}, which is denoted by $\mbox{BS}$, where its objects are all the subsets of Euclidean spaces and its morphisms are all blow-spherical homeomorphisms.

By definition, if $X$ and $Y$ are strongly blow-spherical homeomorphic then they are blow-spherical homeomorphic, but the converse does not hold in general, as we can see in the next example.

\begin{example}
Let $V=\{(x,y)\in\R^2; y=|x|\}$. The mapping $\varphi\colon \R\to V$ given by $\varphi(x)=(x,|x|)$ is a blow-spherical homeomorphism. However, since ${\rm Sing}_1(\R)=\emptyset$ and ${\rm Sing}_1(V)=\{(0,0)\}$, there is no strong blow-spherical homeomorphism $h\colon \R\to V$ and, in particular, $\R$ and $V$ are not strongly blow-spherical homeomorphic.
\end{example}

\begin{proposition}[Proposition 3.12 in \cite{SampaioS:2023}]\label{Lip_implies_bs}
Let $X\subset \mathbb{R}^n$ and $Y\subset \mathbb{R}^m$ be semialgebraic sets. If $h\colon X\rightarrow Y$ is a semialgebraic outer lipeomorphism then $h$ is a blow-spherical homeomorphism. 
\end{proposition}

\begin{example}
$V=\{(x,y)\in\R^2; y^2=x^3\}$ and $W=\{(x,y)\in\R^2; y^2=x^5\}$ are blow-spherical homeomorphic, but are not outer bi-Lipschitz homeomorphic.
\end{example}

\subsection{Tangent Cones}\label{subsec:tg_cones}
Let $X\subset \R^{n+1}$ be an unbounded semialgebraic set (resp. subanalytic set with $p\in \overline{X}$).
We say that $v\in \R^{n+1}$ is a tangent vector of $X$ at infinity (resp. $p$) if there are a sequence of points $\{x_i\}\subset X$ tending to infinity (resp. $p$) and a sequence of positive real numbers $\{t_i\}$ such that 
$$\lim\limits_{i\to \infty} \frac{1}{t_i}x_i= v \quad (\mbox{resp. } \lim\limits_{i\to \infty} \frac{1}{t_i}(x_i-p)= v).$$
Let $C(X,\infty)$ (resp. $C(X,p)$) denote the set of all tangent vectors of $X$ at infinity (resp. $p$). We call $C(X,\infty)$ the {\bf tangent cone of $X$ at infinity} (resp. $p$). 

We have the following characterization.
\begin{corollary}[Corollary 2.16 \cite{FernandesS:2020}]\label{corollary 2.16FernandesS:2020}
Let $X\subset \R^n$ be an unbounded semialgebraic set. Then  
$C(X,\infty)=\{v\in\R^n;\, \exists \gamma:(\varepsilon,+\infty )\to Z$ $C^0$ semialgebraic such that $\lim\limits _{t\to +\infty }|\gamma(t)|=+\infty$ and $\gamma(t)=tv+o_{\infty }(t)\}$, where $g(t)=o_{\infty }(t)$ means $\lim\limits_{t\to +\infty}\frac{g(t)}{t}=0$.
\end{corollary}

Thus, we have the following 
\begin{corollary}[Corollary 2.18 \cite{FernandesS:2020}]\label{dimension_tg_cone}
Let $Z\subset \R^n$ be an unbounded semialgebraic set. Let $\phi:\R^n\setminus\{0\}\to \R^n\setminus\{0\}$ be the semialgebraic mapping given by $\phi(x)=\frac{x}{\|x\|^2}$ and denote $X=\phi(Z\setminus \{0\})$. Then $C(Z,\infty)$ is a semialgebraic set satisfying $C(Z,\infty)=C(X,0)$ and $\dim_{\R} C(Z,\infty) \leq \dim_{\R} Z$.
\end{corollary}

\begin{remark}\label{remarksimplepointcone}
{\rm Another way to present the tangent cone at infinity (resp. $p$) of a subset $X\subset\R^{n+1}$ is via the spherical blow-up at infinity (resp. $p$) of $\R^{n+1}$. In fact, if $X\subset \R^{n+1}$ is a semialgebraic set, then $\partial X'_{\infty}=(C(X,\infty)\cap \mathbb{S}^n)\times \{0\}$ (resp. $\partial X'_{p}=(C(X,p)\cap \mathbb{S}^n)\times \{0\}$).}
\end{remark}

\subsection{Relative multiplicities}

Let $X\subset \R^{m+1}$ be a $d$-dimensional subanalytic subset and $p\in \R^{m+1}\cup \{\infty\}$. We say $x\in\partial X'_p$ is a {\bf  simple point of $\partial X'_p$}, if there is an open subset $U\subset \R^{m+2}$ with $x\in U$ such that:
\begin{itemize}
\item [a)] the connected components $X_1,\cdots , X_r$ of $(X'_p\cap U)\setminus \partial X'_{p}$ are topological submanifolds of $\R^{m+2}$ with $\dim X_i=\dim X$, for all $i=1,\cdots,r$;
\item [b)] $(X_i\cup \partial X'_{p})\cap U$ is a topological manifold with boundary, for all $i=1,\cdots ,r$. 
\end{itemize}
Let ${\rm Smp}(\partial X'_{p})$ be the set of simple points of $\partial X'_{p}$ and we define $C_{\rm Smp}(X,p)=\{t\cdot x; \, t>0\mbox{ and }x\in {\rm Smp}(\partial X'_{p})\}$. Let $k_{X,p}\colon {\rm Smp}(\partial X'_{p})\to \N$ be the function such that $k_{X,p}(x)$ is the number of connected components of the germ $(\rho_{p}^{-1}(X\setminus\{p\}),x)$. 

\begin{remark}\label{remarksimplepointdense}
	It is known that ${\rm Smp}(\partial X'_p)$ is an open dense subset of the $(d-1)$-dimensional part of $\partial X'_p$ whenever $\partial X'_p$ is a $(d-1)$-dimensional subset. where $d=\dim X$ (see \cite{Pawlucki:1985}). 
\end{remark} 
\begin{definition}\label{def:relative_mult}
It is clear the function $k_{X,p}$ is locally constant. In fact, $k_{X,p}$ is constant on each connected component $X_j$ of ${\rm Smp}(\partial X'_{p})$. Then, we define {\bf the relative multiplicity of $X$ at $p$ (along of $X_j$)} to be $k_{X,p}(X_j):=k_{X,p}(x)$ with $x\in X_j$. Let $X_1,...,X_r$ be the connected components of ${\rm Smp}(\partial X'_{p})$. By reordering the indices, if necessary, we assume that $k_{X,p}(X_1)\leq \cdots \leq k_{X,p}(X_r)$. Then we define $k(X,p)=(k_{X,p}(X_1),...,k_{X,p}(X_r))$.
\end{definition}

\begin{remark}
\label{rel_mult_number_of_branches}
Let $X\subset \R^n$ be an unbounded real algebraic curve such that $C(X,\infty)\cap \mathbb{S}^{n-1}=\{a_1,...,a_r\}$. By 
the Conical Structure Theorem at infinity for semialgebraic sets that there exists a constant $R_0\gg 1$ such that for all $R\geq R_0$, we have
$$
X\setminus B_R(0)=\bigcup_{j=1}^r\bigcup_{l=1}^{k_j}{\Gamma_{j,l}}
$$ 
and there is a semialgebraic diffeomorphism $h\colon X\setminus B_R(0)\rightarrow {\rm Cone}_{\infty}(X\cap \mathbb{S}^{n-1}_R(0))$ such that $\|h(x)\|=\|x\|$ and $h|_{X\cap \mathbb{S}^{n-1}_R(0)}=Id$ and, moreover, for each $j\in\{1,...,r\}$, 
$C(\Gamma_{j,l},\infty)\cap \mathbb{S}^{n-1}=\{a_j\}$ for all $l=1,...,k_j$. These $\Gamma_{j,l}$'s are called the {\bf branches of $X$ at infinity}. In particular, $k_j=k_{X,\infty}(a_j,0)$, i.e.,
$k_{X,\infty}(a_i,0)$ is the number of branches of $X$ at infinity that are tangent to $a_i$ at infinity.
\end{remark}

\begin{proposition}[Proposition 3.5 in \cite{SampaioS:2023}]\label{propinvarianciamultiplicidaderelativa}
	Let $X$ and $Y$ be semialgebraic sets in $\mathbb{R}^n$ and $\mathbb{R}^m$ respectively. Let $\varphi\colon X\rightarrow Y$ be a blow-spherical homeomorphism at $p\in \mathbb{R}^n\cup \{\infty\}$. Then  
	$$k_{X,p}(x)=k_{Y,q}(\varphi'(x)),$$
	for all $x\in Smp(\partial X'_p)\subset \partial X'_p$, where $q=\lim\limits_{x\to p}\varphi(x)$. In particular, $k(X,p)=k(Y,q)$.
\end{proposition}

\begin{proposition}[Proposition 3.6 in \cite{SampaioS:2023}]\label{coneinvarianteblow}
 If $h\colon X\rightarrow Y$ is a blow-spherical homeomorphism at $p\in X\cup \{\infty\}$, then $C(X,p)$ and $C(Y,h(p))$ are blow-spherical homeomorphic at $0$.
\end{proposition}

\section{On the classification of real algebraic curves}
In this section, we present some examples and results about the classification of real algebraic curves under blow-spherical homeomorphisms.

\begin{proposition}\label{propblowspherical}
Let $X,\widetilde X\subset \R^n$ be two real semialgebraic curves. 
Then the following statements are equivalent:
\begin{enumerate}
\item [(1)] $X$ and $\widetilde X$ are blow-spherical homeomorphic at infinity;
\item [(2)] $k(X,\infty)=k(\widetilde X,\infty)$;
\item [(3)] $X$ and $\widetilde X$ are strongly blow-spherical homeomorphic at infinity;
\end{enumerate}
\end{proposition}
\begin{proof}
Assume that $C(X,\infty)\cap \mathbb{S}^{n-1}=\{a_1,...,a_r\}$ and $C(\widetilde{X},\infty)\cap \mathbb{S}^{n-1}=\{\widetilde{a}_1,...,\widetilde{a}_s\}$.
By Remark \ref{remarksimplepointdense}, ${\rm Smp}(\partial X'_{\infty})$ is an open dense subset of the $0$-dimensional part of $\partial X'_{\infty}$, hence ${\rm Smp}(\partial X'_{\infty})=\partial X'_{\infty}$.

Clearly, (3) $\Rightarrow$ (1).

(1) $\Rightarrow$ (2).
Assume that $X$ and $\widetilde{X}$ are blow-spherical homeomorphic at infinity. By Proposition \ref{propinvarianciamultiplicidaderelativa},  $k(X,\infty)=k(\widetilde X,\infty)$.

(2) $\Rightarrow$ (3). Assume that $k(X,\infty)=k(\widetilde X,\infty)$.
Thus, $r=s$ and by reordering the indices, if necessary, we may assume that $k_i=k_{X,\infty}(a_i,0)=k_{\widetilde{X},\infty}(\widetilde{a_i},0)$ for all $i=1,...,r$ and $k_1\leq k_2\leq ....\leq k_r$. Thus, it follows from Remark \ref{rel_mult_number_of_branches} and the Conical Structure Theorem at infinity for semialgebraic sets that there exists a constant $R_0\gg 1$ such that for all $R\geq R_0$, we have
$$
X\setminus B_R(0)=\bigcup_{j=1}^r\bigcup_{l=1}^{k_j}{\Gamma_{j,l}}
$$ 
and there is a semialgebraic diffeomorphism $h\colon X\setminus B_R(0)\rightarrow {\rm Cone}_{\infty}(X\cap \mathbb{S}^{n-1}_R(0))$ such that $\|h(x)\|=\|x\|$ and $h|_{X\cap \mathbb{S}^{n-1}_R(0)}=Id$ and, moreover, for each $j\in \{1,...,r\}$, $C(\Gamma_{j,l},\infty)\cap \mathbb{S}^{n-1}=\{a_j\}$ for all $l=1,...,k_j$. In particular, we consider $h|_{\Gamma_{j,l}}\colon\Gamma_{j,l}\rightarrow {\rm Cone}_{\infty}(\Gamma_{j,l}\cap \mathbb{S}^{n-1}_R(0))$. 

Now define the curve $\alpha_{j,l}\colon [R,+\infty)\rightarrow {\rm Cone}_{\infty}(\Gamma_{j,l}\cap \mathbb{S}^{n-1}_R(0))$ given by $\alpha_{j,l}(t)=t\frac{a_{j,l}}{|a_{j,l}|}$, where $\{a_{j,l}\}=\Gamma_{j,l}\cap \mathbb{S}^{n-1}_R(0)$. Thus, we define the curve $\beta_{j,l}\colon [R,+\infty)\rightarrow \Gamma_{j,l}$ by $\beta_{j,l}(t):=(h^{-1}\circ \alpha_{j,l})(t)$.

Analogously, we also have
$$
\widetilde{X}\setminus B_R(0)=\bigcup_{j=1}^r\bigcup_{l=1}^{k_j}{\widetilde{\Gamma}_{j,l}},
$$ 
 and there are semialgebraic diffeomorphisms $\widetilde{h}\colon \widetilde{X}\setminus B_R(0) \rightarrow {\rm Cone}_{\infty}(\widetilde{X}\cap \mathbb{S}^{n-1}_R(0))$, $\widetilde{\alpha}_{j,l}\colon [R,+\infty)\rightarrow {\rm Cone}_{\infty}(\widetilde{\Gamma}_{j,l}\cap \mathbb{S}^{n-1}_R(0))$ and $\widetilde{\beta}_{j,l}\colon [R,+\infty)\rightarrow \widetilde{\Gamma}_{j,l}$ and, moreover, for each $j$, $C(\widetilde{\Gamma}_{j,l},\infty)\cap \mathbb{S}^{n-1}=\{\widetilde{a}_j\}$ for all $l=1,...,k_j$.

Let $A=X\setminus B_R(0)$ and $\widetilde{A}=\widetilde{X}\setminus B_R(0)$ and define $\varphi\colon A\rightarrow \widetilde{A}$ by $\varphi(z)=\widetilde{\beta}_{j,l}\circ \beta^{-1}_{j,l}(z)$ if $z\in \Gamma_{j,l}$.
We have that $\varphi$ is a strong blow-spherical homeomorphism at infinity and $\varphi'\colon X'_{\infty}\rightarrow \widetilde{X}'_{\infty}$ is given by
$$
\varphi'(x,s)=\left\{\begin{array}{ll}
	\left(\frac{\varphi(\frac{x}{s})}{\|\varphi(\frac{x}{s})\|},\frac{1}{\|\varphi(\frac{x}{s})\|}\right),& \mbox{ if }s\not =0;\\
	(\widetilde{a}_j,0),& \mbox{ if } (x,s)=(a_j,0).
\end{array}\right.
$$
In order to see that, it is enough to prove that $\varphi$ is a blow-spherical homeomorphism at infinity. Let us prove that $\varphi'$ is continuous at each $(a_j,0)\in \partial A'_{\infty}$. Thus take $(a_j,0)\in \partial A'_{\infty}$ and consider a sequence $(x_k,s_k)_{k\in \mathbb{N}}\subset A'_{\infty}\setminus \partial A'_{\infty}$ such that $(x_k,s_k)\rightarrow (a_j,0)$. Thus, for any subsequence $\{(x_{k_i},s_{k_i})\}_{i\in \mathbb N}$, there is some $l\in \{1,...,k_j\}$ and a subsequence $\{(z_k,t_k)\}_{k\in \mathbb N}\subset\{(x_{k_i},s_{k_i})\}_{i\in \mathbb N}$ such that $\frac{z_k}{t_k}\in \Gamma_{j,l}$ for all $k$.
Since \begin{eqnarray*}
	\|\beta_{j,l}(t)\| & = & \|h^{-1}\circ \alpha_{j,l}(t)\|\\
	& = & \left\|h^{-1}\left(t\frac{a_{j,l}}{\|a_{j,l}\|}\right)\right\|\\
	& = & t,
\end{eqnarray*}
we have 
$$
\frac{z_k}{t_k}=\beta_{j,l}\left(\frac{1}{t_k}\right).
$$
Therefore, 
\begin{eqnarray*}
	\varphi'(z_k,t_k) & = & \left( \frac{\widetilde{\beta}_{j,l}(\frac{1}{t_k})}{\|\widetilde{\beta}_{j,l}(\frac{1}{t_k})\|},  \frac{1}{\|\widetilde{\beta}_{j,l}(\frac{1}{t_k})\|}\right).
\end{eqnarray*}
Since $\lim\limits_{t\to 0^+}\frac{\widetilde{\beta}_{j,l}(\frac{1}{t})}{\|\widetilde{\beta}_{j,l}(\frac{1}{t})\|}=\widetilde{a}_j$, we have
$\lim\limits_{k\rightarrow +\infty}{\varphi'(z_k,t_k)}=(\widetilde{a}_j,0).$
This shows that
$\lim\limits_{k\rightarrow +\infty}{\varphi'(x_k,s_k)}=(\widetilde{a}_j,0).$

Thus, $\varphi'$ is continuous at each $(a_j,0)$. Analogously, $(\varphi^{-1})'$ is continuous at each $(\widetilde{a}_j,0)\in \partial \widetilde{A}'_{\infty}$. Therefore, $\varphi$ is a strong blow-spherical homeomorphism at infinity. 
\end{proof}

\subsection{Normal forms for the classification at infinity}\label{sec:normal_forms}

Let $\mathcal{F}(\{0,1\};\mathbb{Z}_{\geq 0})$ be the set of all non-null  functions from $\{0,1\}$ to $\mathbb{Z}_{\geq 0}$. For each positive integer number $N$, let $\mathcal{A}_N$ be the subset of $(\mathcal{F}(\{0,1\};\mathbb{Z}_{\geq 0}))^N$ formed by all $(r_1,...,r_N)$ satisfying the following:
\begin{enumerate}
\item $r_l(0)\leq r_{l+1}(0)$ for all $l\in \{1,\cdots, N-1\}$; 
\item If $r_l(0)=r_{l+1}(0)$ then $r_l(1)\leq r_{l+1}(1)$.
\end{enumerate}
Let $\mathcal{A}=\bigcup\limits_{N=1}^{\infty}\mathcal{A}_N$.
Let $A=(r_1,...,r_N)\in \mathcal A$. For $j\in \{0, 1\}$ and $r_l(j)>0$, we define the following curves:

\begin{equation*}\label{def.normalform}
X_{A,1}=\left\{(x,y)\in \mathbb{R}^2; \prod_{l=1}^{N}\prod_{r=1}^{r_l(1)}{( (y-lx)^{2}-r(y+lx))}=0 \right\},
\end{equation*}
and 
\begin{equation*}\label{def.normalform_two}
	X_{A,0}=\left\{(x,y)\in \mathbb{R}^2; \prod_{l=1}^{N}\prod_{r=1}^{r_l(0)}{( (y-lx)-r)}=0 \right\}.
\end{equation*}
 Moreover, if $r_l(j)=0$ we define $X_{A,j}=\{0\}$.  Finally, we define the realization of $A=(r_1,...,r_N)$ to be the curve $X_A:=X_{A,0}\cup X_{A,1}$.

Thus, it follows from Proposition \ref{propblowspherical} and the definition of $A$, the following classification result:

\begin{proposition}\label{thmnormalform}
	For each real algebraic curve $X\subset \mathbb{R}^n$, there exists a unique $A\in \mathcal{A}$ such that $X_A$ and $X$ are blow-spherical homeomorphic at
	infinity.
\end{proposition}
\begin{proof}
We consider $\overline{X}$ the projective closure of $X$ and $\{c_1,\cdots,c_N\}=\overline{X}\cap L_{\infty}$, where $L_{\infty}$ is the hyperplane at infinity. 

By taking local charts, it follows from Lemma 3.3.5 in \cite{Milnor:1968} that there exist an open neighborhood $V_{l}\subset \mathbb{RP}^n$ of $c_l$ and $\Upsilon_{l,1},\cdots, \Upsilon_{l,r_l}$ such that $\Upsilon_{l,i}\cap \Upsilon_{l,i'}=\{c_l\}$ whenever $i\neq i'$ and 
	$$Y\cap V_l=\bigcup_{i=1}^{r_l}{\Upsilon_{l,i}}$$
	and, moreover, for each $i\in \{1,\cdots, r_l\}$, there exists an analytic homeomorphism $\upsilon_{l,i}\colon (-\epsilon,\epsilon)\rightarrow \Upsilon_{l,i}$ with $\upsilon_i(0)=c_l$.
By shrinking  $V_{l}$, if necessary, we may assume that for each $i\in \{1,\cdots, s_l\}$, $\Upsilon_{l,i}\subset L_{\infty}$ or $\Upsilon_{l,i}\cap L_{\infty}=\{c_l\}$.
By reordering the indices, if necessary, there is $r_l> 0$ such that $\Upsilon_{l,i}\cap L_{\infty}=\{c_l\}$ for all $i\in \{1,\cdots, r_l\}$ and $\Upsilon_{l,i}\subset L_{\infty}$ for all $i\in \{r_l+1,\cdots, s_l\}$. By reordering the indices again, if necessary, we may assume that $r_l\leq r_{l+1}$, for all $j\in \{1,...,N\}$.

We denote the half-branch by $\Upsilon_{l,i}^{+}=\upsilon_{l,i}(0,\epsilon)$ and $\Upsilon_{l,i}^{-}=\upsilon_{l,i}(-\epsilon,0)$. So, denoting $\Gamma_{l,i}^+=\Upsilon_{l,i}^{+}\cap \R^n$ and $\Gamma_{l,i}^-=\Upsilon_{l,i}^{-}\cap \R^n$, we obtain that 
	$$\{\Gamma_{l,1}^+,\Gamma_{l,1}^-,\cdots, \Gamma_{l,r_l}^+,\Gamma_{l,r_l}^-\}_{l=1}^{N}
	$$ 
are all the branches of $X$ at infinity. For each $l\in \{1,..,N\}$, there is $a_l\in \mathbb{S}^{n-1}$ such that $\pi^{-1}(c_l)\cap \mathbb{S}^{n-1}=\{-a_l,a_l\}$, where $\pi\colon \R^n\setminus\{0\}\to \mathbb{RP}^{n-1}\cong L_{\infty}$ is the canonical projection.
Thus, for each $l\in \{1,..,N\}$, $C(\Gamma_{l,i},\infty)\cap \mathbb{S}^{n-1}\subset\{- a_l,a_l\}$. 

For each $l\in \{1,..,N\}$, by reordering the indices, if necessary, there are non-negative integer numbers $n_l, p_l$ and $z_l$ such that:
\begin{itemize}
 \item $C(\Gamma_{l,i}^+,\infty)=-C(\Gamma_{l,i}^-,\infty)$ for all $i\in \{1,...,z_l\}$;
 \item $C(\Gamma_{l,i}^+,\infty)\cap \mathbb{S}^{n-1}=C(\Gamma_{l,i}^-,\infty)\cap \mathbb{S}^{n-1}=\{- a_l\}$,  for all $i\in \{z_l+1,...,z_l+n_l\}$;
 \item $C(\Gamma_{l,i}^+,\infty)\cap \mathbb{S}^{n-1}=C(\Gamma_{l,i}^-,\infty)\cap \mathbb{S}^{n-1}=\{ a_l\}$, for all $i\in \{z_l+n_l+1,...,r_l=z_l+n_l+p_l\}$.
\end{itemize}

By changing $a_l$ by $-a_l$, if necessary, we may assume that $n_l\leq p_l$.
Thus, we define $\tilde\Gamma_{l,i}^+=\Gamma_{l,i}^+$ for all $i\in \{1,...,r_l\}$ and
$$
\tilde \Gamma_{l,i}^-=\left\{\begin{array}{ll}
\Gamma_{l,i}^-,\mbox{ if }i\in \{1,...,z_l\}\cup \{z_l+n_l+1,...,r_l\}\\
\Gamma_{l,i+n_l}^-,\mbox{ if }i\in \{z_l+1,...,z_l+n_l\}.
\end{array}\right.
$$
Thus, we have the following:
\begin{itemize}
 \item $C(\tilde\Gamma_{l,i}^+,\infty)=-C(\tilde\Gamma_{l,i}^-,\infty)$ for all $i\in \{1,...,z_l+n_l\}$;
 \item $C(\tilde\Gamma_{l,i}^+,\infty)=C(\tilde\Gamma_{l,i}^-,\infty)\}$,  for all $i\in \{z_l+n_l+1,...,r_l\}$.
\end{itemize}

For each $l\in \{1,..,N\}$, we define $r_l\colon\{0,1\}\to \mathbb{N}$ by $r_l(0)=z_l+n_l$ and $r_l(1)=p_l$.

By reordering the indices, if necessary, we may assume that 
\begin{enumerate}
\item $r_l(0)\leq r_{l+1}(0)$ for all $l\in \{1,\cdots, N-1\}$; 
\item If $r_l(0)=r_{l+1}(0)$ then $r_l(1)\leq r_{l+1}(1)$.
\end{enumerate}
Therefore, $A={(r_1,...,r_N)}\in \mathcal{A}$ and by Proposition \ref{propblowspherical}, we have that $X_A$ is blow-spherical homeomorphic at infinity to $X$. The uniqueness of $A$ follows from the definition of $\mathcal{A}$ and Proposition \ref{propblowspherical}.
\end{proof}
\begin{remark}\label{remark:spatial}
Proposition \ref{thmnormalform} says in particular that any spacial real algebraic curve is blow-spherical homeomorphic at infinity to a plane real algebraic curve.
\end{remark}
Note that Remark \ref{remark:spatial} is not true in the global case, as it is shown in the next example.

\begin{example}
	Let $Y$ and $Z$ be the algebraic curves in $\mathbb{R}^3$ given by
	$$Y=\textstyle{\left\{(x,y,z)\in \mathbb{R}^3;(y-1)(y-\frac{x^2}{2})(y-x^2)(y-2x^2)=0\mbox{ and }z=0\right\}}$$
	and 
	$$Z=\textstyle{\left\{(x,y,z)\in \mathbb{R}^3;\left((x-\frac{1}{2})^2+(y-\frac{1}{2})^2+z^2-\frac{1}{2}\right)=0\mbox{ and } x=y\right\}}.$$
Thus the spatial algebraic curve $X=Y\cup Z$ is not blow-spherical equivalent to a plane curve (see figure \ref{fig:spatial}). 
\end{example}

\begin{figure}[H]
	\centering
	\includegraphics[scale=0.65]{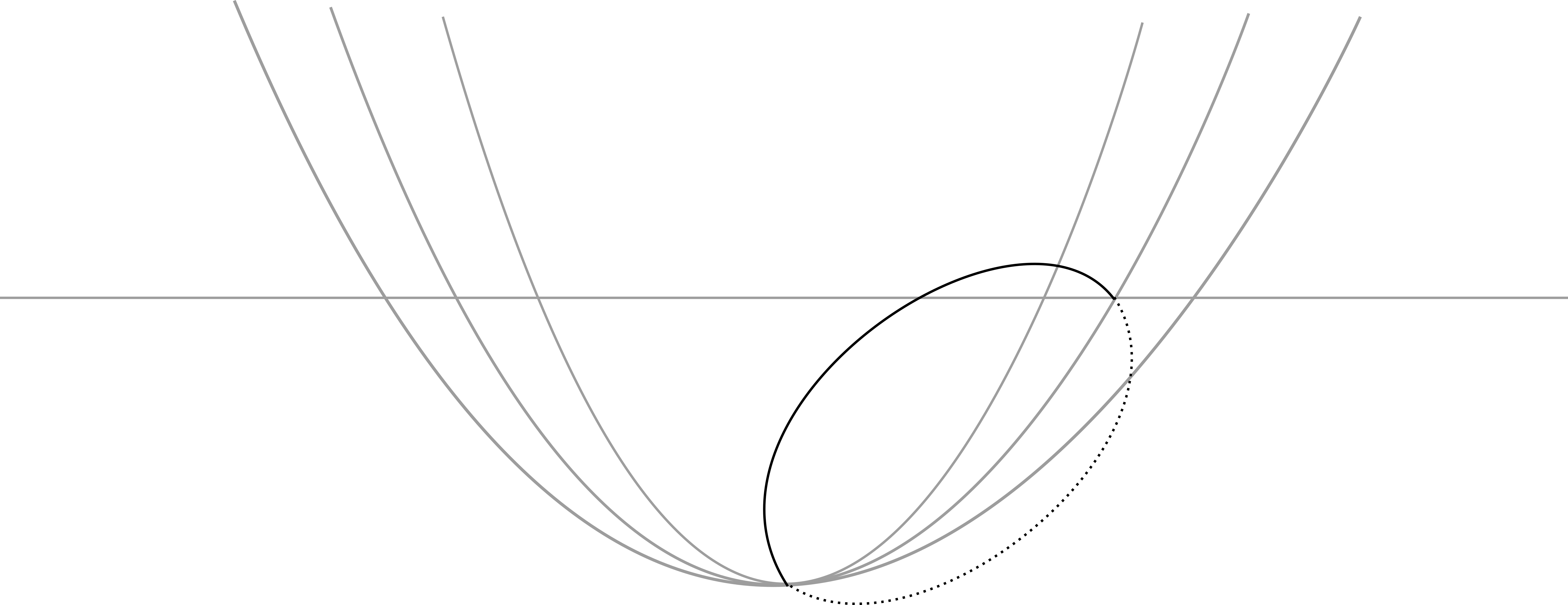}
	\caption{The spatial algebraic curve $X$}
	\label{fig:spatial}
\end{figure}

\subsection{Realization of the invariant $k(\cdot,\infty)$}\label{sec:realization}

\begin{definition}
	For each positive integer $n$, let $\mathcal{N}_n$ be the set of all $(\eta_1, \eta_2,\cdots,\eta_n)\in (\mathbb{Z}_{>0})^n$ such that $\eta _1\leq \eta _2\leq \cdots \leq \eta_n$. Let $\mathcal{N}=\bigcup\limits_{n=1}^{\infty} \mathcal{N}_n$.
\end{definition}
%
By definition, we have that if $X\subset \mathbb{R}^n$ is a semialgebraic curve then $k(X,\infty)\in\mathcal{N}$. Reciprocally, for $\eta=(\eta_1, \eta_2,\cdots,\eta_N)\in \mathcal{N}$, it is easy to find a semialgebraic curve $X\subset \mathbb{R}^n$ such that $k(X,\infty)=\eta$. For example,  
\begin{equation*}
X=\bigcup_{l=1}^{N}\bigcup_{r=1}^{\eta_l}\left\{(x,y)\in \mathbb{R}^2; y-lx-r=0 \mbox{ and }y+lx\geq 0\right\}.
\end{equation*}
\begin{definition}
	We say that $\eta=(\eta_1, \eta_2,\cdots,\eta_r)\in \mathcal{N}$ is {\bf algebraically realizable}, if there exists a real algebraic curve $X\subset \mathbb{R}^n$ such that $k(X,\infty)=\eta$.
\end{definition}
The next result gives a necessary and sufficient condition for a $\eta=(\eta_1, \eta_2,\cdots,\eta_n)\in \mathcal{N}$ to be algebraically realizable. Let $v=(v_1,\cdots,v_n)\in \mathbb{R}^n$ a vector, we denote by $\|\cdot\|_1$ the norm
 given by $\|v\|_1=|v_1|+\cdots+|v_n|$. 
\begin{proposition}\label{thm:realizable}
	$\eta=(\eta_1,\cdots,\eta_k)\in \mathcal{N}$ is algebraically realizable if and only if $\|\eta\|_{1} \equiv 0 \pmod{2}$.
\end{proposition}
\begin{proof}
Assume that $\eta=(\eta_1,\cdots,\eta_k)\in \mathcal{N}$ is algebraically realizable, that is, there exists an unbounded real algebraic curve $X\subset \mathbb{R}^n$ such that $\eta=k(X,\infty)$. 
We consider $\overline{X}$ the projective closure of $X$ and $\{c_1,\cdots,c_N\}=\overline{X}\cap L_{\infty}$, where $L_{\infty}$ is the hyperplane at infinity. 

By the proof of Proposition \ref{thmnormalform}, we obtain that there exist an open neighbourhood $V_{l}\subset \mathbb{RP}^n$ of $c_l$ and $\Upsilon_{l,1},\cdots, \Upsilon_{l,r_l}$ such that $\Upsilon_{l,i}\cap \Upsilon_{l,i'}=\{c_l\}$ whenever $i\neq i'$, 
	$$Y\cap V_l=\bigcup_{i=1}^{r_l}{\Upsilon_{l,i}}$$
	and, moreover, for each $i\in \{1,\cdots, r_l\}$, there exists an analytic homeomorphism $\upsilon_{l,i}\colon (-\epsilon,\epsilon)\rightarrow \Upsilon_{l,i}$ with $\upsilon_i(0)=c_l$.
Additionally, there is $r_l> 0$ such that $\Upsilon_{l,i}\cap L_{\infty}=\{c_l\}$ for all $i\in \{1,\cdots, r_l\}$ and $\Upsilon_{l,i}\subset L_{\infty}$ for all $i\in \{r_l+1,\cdots, s_l\}$. By reordering the indices again, if necessary, we may assume that $r_l\leq r_{l+1}$, for all $j\in \{1,...,N\}$.

We denote the half-branch by $\Upsilon_{l,i}^{+}=\upsilon_{l,i}(0,\epsilon)$ and $\Upsilon_{l,i}^{-}=\upsilon_{l,i}(-\epsilon,0)$. So, denoting $\Gamma_{l,i}^+=\Upsilon_{l,i}^{+}\cap \R^n$ and $\Gamma_{l,i}^-=\Upsilon_{l,i}^{-}\cap \R^n$, we obtain that 
	$$\{\Gamma_{l,1}^+,\Gamma_{l,1}^-,\cdots, \Gamma_{l,r_l}^+,\Gamma_{l,r_l}^-\}_{l=1}^{N}
	$$ 
are all the branches of $X$ at infinity. For each $l\in \{1,..,k\}$, there is $b_l\in \mathbb{S}^{n-1}$ such that $\pi^{-1}(c_l)\cap \mathbb{S}^{n-1}=\{-b_l,b_l\}$, where $\pi\colon \R^n\setminus\{0\}\to \mathbb{RP}^{n-1}\cong L_{\infty}$ is the canonical projection. Thus, for each $l\in \{1,..,k\}$, $C(\Gamma_{l,i},\infty)\cap \mathbb{S}^{n-1}\subset\{- b_l,b_l\}$.
By Remark \ref{rel_mult_number_of_branches}, 
$$
k_{X,\infty}(b_l,0)+k_{X,\infty}(-b_l,0)=2r_l\equiv 0 \pmod{2},
$$
where $k_{X,\infty}(v)$ is defined to be zero if $v\not\in Smp(\partial X')$.

Assume that $C(X,\infty)\cap \mathbb{S}^{n-1}=\{a_1,...,a_k\}$ and $\eta_l=k_{X,\infty}(a_l,0)$ for all $l\in \{1,...,k\}$.

We consider the following decomposition of $\{1,...,k\}$: 
$$
\{l_1,...,l_{2s}\}=\{l\in \{1,...,k\};- a_l,a_l\in C(X,\infty)\}
$$ 
and 
$$
\{l_{2s+1},...,l_k\}= \{1,...,k\}\setminus \{l_1,...,l_{2s}\}
$$
and such that $a_{l_i}=-a_{l_{i+s}}$ for all $i\in \{1,...,s\}$.
	Therefore, by writing $k_{X,\infty}(a)$ instead of $k_{X,\infty}(a,0)$, we have
	\begin{eqnarray*}
		\|\eta\|_1& = & k_{X,\infty}(a_{1_1})+k_{X,\infty}(a_{l_{1+s}})+\cdots +k_{X,\infty}(a_{1_s})+k_{X,\infty}(a_{l_{2s}})\\
		& &+ k_{X,\infty}(a_{1_{2s+1}})+k_{X,\infty}(-a_{1_{2s+1}})+\cdots +k_{X,\infty}(a_{1_k}))+k_{X,\infty}(-a_{l_{k}})\\
		& = & 2r_{l_1} +\cdots + 2r_{l_s}+2r_{l_{2s+1}} +\cdots + 2r_{l_k} \\
		& \equiv & 0 \pmod{2}.
	\end{eqnarray*}

For the converse, assume that $\eta=(\eta_1,\cdots,\eta_k)\in \mathcal{N}$ satisfies $\|\eta\|_{1} \equiv 0 \pmod{2}$. Thus, $\#\{j;\eta_j \equiv 1 \pmod{2}\}\equiv 0 \pmod{2}$. Let $J=\{j_1\leq ... \leq j_{2m}\}\subset I:=\{1,...,k\}$ be the indices such that $\eta_{j}\equiv 1 \pmod{2}$ if and only if $j\in J$. For each $j_i\in J$, let $n_i$ be the non-negative integer number such that $\eta_{j_i}=2n_i+1$. Let $I\setminus J=\{q_1,...,q_N\}$. We consider the following three algebraic curves
\begin{equation*}
X^+=\left\{(x,y)\in \mathbb{R}^2; \prod_{l=1}^{m}(y-lx)\prod_{r=1}^{n_l}{((y-lx)^{2}-r(y+lx))}=0 \right\},
\end{equation*}
\begin{equation*}
X^-=\left\{(x,y)\in \mathbb{R}^2; \prod_{l=1}^{m}\prod_{r=1}^{n_{m+l}}{((y-lx)^{2}+r(y+lx))}=0 \right\}
\end{equation*}
and
\begin{equation*}
X^0=\left\{(x,y)\in \mathbb{R}^2; \prod_{l=1}^{N}\prod_{r=1}^{\eta_{q_l}}{((y-(m+l)x)^{2}-r(y+(m+l)x))}=0 \right\}
\end{equation*}
Then, for $X=X_0\cup X^+\cup X^-$, we have that $k(X,\infty)=\eta$.

\end{proof}

\subsection{Some considerations on the global case}
In the global case, the problem of classification is harder. For instance, two homeomorphic algebraic curves having the same relative multiplicities may not be blow-spherical homeomorphic, as we can see in the next example. 

\begin{example}
Let 
$$\textstyle{X=\left\{(x,y)\in \mathbb{R}^2;p(x,y)\left(x^2-y^2-y^3\right) \left((x+\frac{7\sqrt{12}}{8})^2+(y-3)^2-\frac{76}{64}\right)=0\right\}}$$ 
and 
$$\textstyle{\widetilde{X}=\left\{(x,y)\in \mathbb{R}^2;q(x,y)\left(x^2-y^2-y^3\right) \left((x+\frac{7\sqrt{12}}{8})^2+(y-3)^2-\frac{76}{64}\right)=0\right\}},$$
where $p$ and $q$ are the polynomials given by $p(x,y)=((x-\frac{7\sqrt{12}}{8})^2+(y-3)^2-\frac{76}{64})$ and $q(x,y)=((x+\frac{27\sqrt{80}}{28})^2+(y-5)^2-\frac{864}{784})$. Then $X$ and $\widetilde{X}$
are homeomorphic algebraic curves and there is a bijection $\sigma\colon {\rm Sing}(X)\cup \{\infty\}\to {\rm Sing}(\widetilde{X})\cup \{\infty\}$ such that $\sigma(\infty)=\infty$ and $k(X,p)=k(\widetilde{X},\sigma(p))$ for all $p\in {\rm Sing}(X)\cup \{\infty\}$. However, $X$ and $\widetilde{X}$ are not blow-spherical homeomorphic (see Figures \ref{fig1} and \ref{fig2}).
\end{example}

\begin{figure}[h!]
	\begin{tabular}{cc}
		\begin{minipage}[c][2.5cm][c]{5.5cm}
			\centering \includegraphics[scale=0.07]{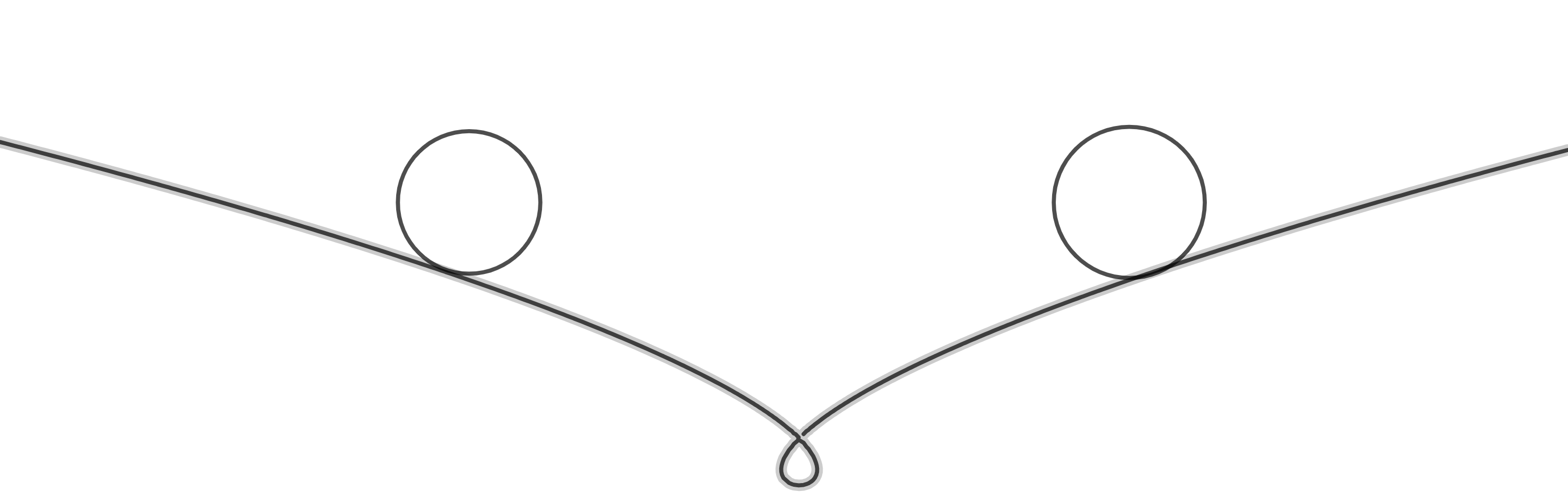}
			\caption{Curve $X$}\label{fig1}
		\end{minipage}
		&
		\begin{minipage}[c][2.5cm][c]{5.5cm}
			\centering \includegraphics[scale=0.07]{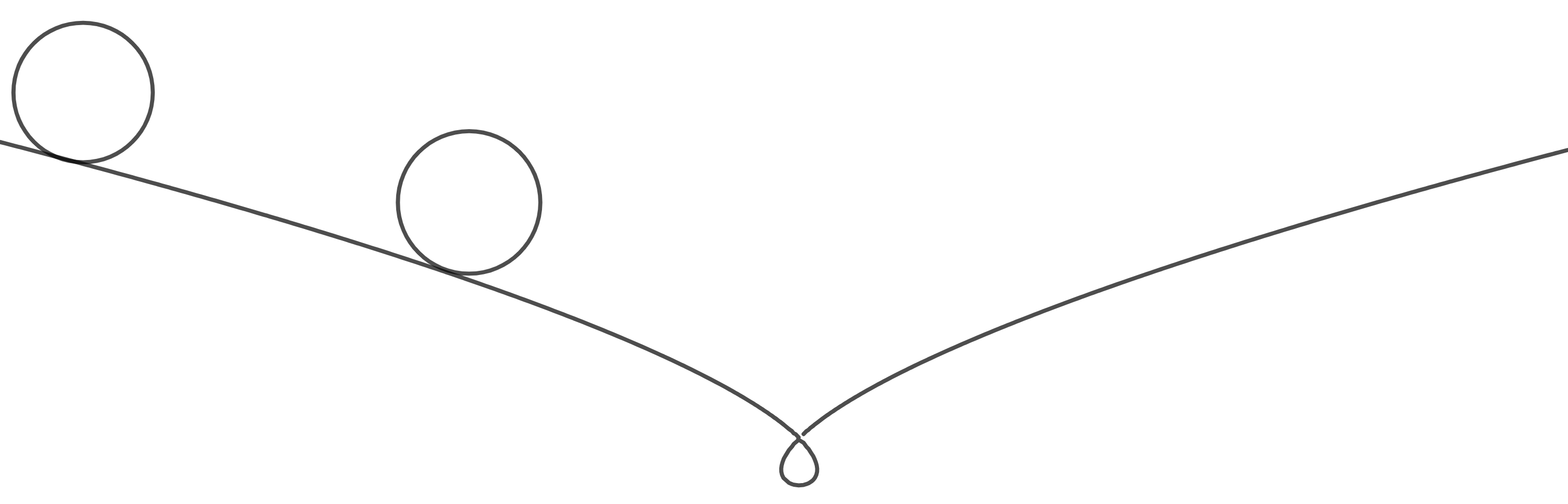}
			\caption{Curve $\widetilde{X}$}\label{fig2}
		\end{minipage}
	\end{tabular}
\end{figure}

Thus, we need of the following notion:

\begin{definition}
We say that a homeomorphism $\phi\colon X\to \widetilde{X}$ between two analytic curves is a {\bf tangency-preserving homeomorphism} if, for each $p\in X\cup \{\infty\}$, two half-branches $\Gamma_1$ and $\Gamma_2$ of $(X,p)$ are tangent at $p$ if and only if $\phi(\Gamma_1)$ and $\phi(\Gamma_2)$ are tangent at $\phi(p)$.
\end{definition}

\begin{proposition}
	Let $X,\widetilde X\subset \R^n$ be two connected  real algebraic curves. Then the following statements are equivalent:
	\begin{enumerate}
		\item [(1)] $X$ and $\widetilde X$ are blow-spherical homeomorphic;
		\item [(2)] There is a tangency-preserving homeomorphism $\phi\colon X\to \widetilde X$;
		\item [(3)] $X$ and $\widetilde{X}$ are strongly blow-spherical homeomorphic.
	\end{enumerate}
\end{proposition}
\begin{proof}
It follows from Propositions \ref{propinvarianciamultiplicidaderelativa} and \ref{coneinvarianteblow} that $(1)\Rightarrow (2)$. 

Since $(3)\Rightarrow (1)$, we only have to prove $(2)\Rightarrow (3)$. Assume that there is a tangency-preserving homeomorphism $\phi\colon X\to \widetilde X$.

Let ${\rm Sing}_1(X)=\{p_1,...,p_r\}$ and ${\rm Sing}_1(\widetilde X)=\{\widetilde p_1,...,\widetilde p_s\}$. Since $\phi$ is a tangency-preserving homeomorphism, it follows from \cite[Proposition 6.9]{Sampaio:2022b} that $\phi({\rm Sing}_1(X))={\rm Sing}_1(\widetilde X)$ and, in particular, $r=s$. Thus, we assume that $\phi(p_i)=\widetilde p_i$ for all $i\in \{1,...,r\}$.

By following the proof of Theorem 6.19 in \cite{Sampaio:2022b}, we can find $\epsilon>0$ and a strong blow-spherical homeomorphism $\varphi\colon \bigcup\limits_{i=1}^r (X\cap B_{\epsilon}(p_i))\to \bigcup\limits_{i=1}^r(\widetilde X\cap B_{\epsilon}(\widetilde p_i))$ such that for each $i\in \{1,...,r\}$, $\|\varphi(x)-\widetilde p_i\|=\|x-p_i\|$ for all $x\in X\cap B_{\epsilon}(p_i)$.

It follows from Proposition \ref{propblowspherical} that there exists a strong blow-spherical homeomorphism $h\colon X\setminus B_{R}(0)\rightarrow \widetilde{X}\setminus B_{R}(0)$ such that $\|h(x)\|=\|x\|$ for all $x\in X\setminus B_{R}(0)$. 

Now, we define the following set 
$$X_{\epsilon, R}=\left( X\cap B_{2R}(0)\right)\setminus \left\{B_{\epsilon/2}(p_1)\cup \cdots \cup B_{\epsilon/2}(p_s)\right\},$$
and
$$\widetilde{X}_{\epsilon, R}=\left(\widetilde{X}\cap B_{2R}(0)\right)\setminus \left\{B_{\epsilon/2}(p_1)\cup \cdots \cup B_{\epsilon/2}(p_s)\right\}.$$

Note that $X_{\epsilon,R}$ (resp. $\widetilde{X}_{\epsilon, R}$) is a finite union of compact one-dimensional with boundary, i.e, $X_{\epsilon,R}=\bigcup_{i=1}^{l_0}{X_{\epsilon,R}^i}$ (resp. $\widetilde{X}_{\epsilon,R}=\bigcup_{i=1}^{l_0}{\widetilde{X}_{\epsilon,R}^i}$), where each $X_{\epsilon,R}^i$ (resp. $\widetilde{X}_{\epsilon,R}^i$) is diffeomorphic to the compact interval $[0,1]$. Thus there is a diffeomorphism $\Phi\colon X_{\epsilon,R}\rightarrow \widetilde{X}_{\epsilon, R}$ such that each $\Phi(X_{\epsilon,R}^i)$ is contained in the connected component of $\widetilde{X}\setminus {\rm Sing}_1(\widetilde{X})$ which contains $\phi(X_{\epsilon,R}^i)$. By using standard arguments of bump functions, we may define a strong blow-spherical homeomorphism $F\colon X\rightarrow \widetilde{X}$ such that 

\[   
F(x) = 
\begin{cases}
	\varphi(x), &\quad\text{if} \ x\in X; \|x-p_j\|\leq \epsilon/2\\
	h(x), &\quad\text{if} \ x\in E_i; \|x\|\geq 2R\\ 
	\Phi(x), &\quad\text{if} \ x\in X_{2\epsilon, R/2},
\end{cases}
\]
which finishes the proof.
\end{proof}


\end{document}